\documentclass[a4paper,12pt]{amsart}
\usepackage{amsmath,amsthm}
\usepackage{amsfonts,amsmath,amsthm, amssymb, mathrsfs}
\usepackage{float}
\usepackage{euscript,eufrak,verbatim}
\usepackage{graphicx}
\usepackage[usenames]{color}
\usepackage[colorlinks,linkcolor=red,anchorcolor=blue,citecolor=blue]{hyperref}
\usepackage{amsmath}
\usepackage{amsthm}
\usepackage[all]{xy}
\usepackage{graphicx} 
\usepackage{amssymb}

\newtheorem{thm}{Theorem}[section]
\newtheorem{prop}[thm]{Proposition}
\newtheorem{df}[thm]{Defintion}
\newtheorem{lem}[thm]{Lemma}

\newtheorem{ex}[thm]{Example}

\newtheorem*{thm1}{Theorem A}

\def\N{\mathbb{N}}

\def\N{\mathbb{N}}

\def\XX{\mathcal{X}}

\def\EE{\mathcal{E}}

\def\diam{\text{\rm diam}}

\def\mdim{\text{\rm mdim}}

\def\mdim{\text{\rm mdim}}
\def\over{\overline { \rm mdim}_{M}(T,\XX,d)}
\def\under{\underline{ \rm mdim}_{M}(T,\XX,d)}

\def\logf{\log\frac{1}{\epsilon}}
\makeatletter 
\@addtoreset{equation}{section}
\numberwithin{equation}{section}

\title{Some notes on variational principle for metric mean dimension}

\author{Rui Yang$^1$, Ercai Chen$^1$ and Xiaoyao Zhou*$^1$
}
\address
{1.School of Mathematical Sciences and Institute of Mathematics, Nanjing Normal University, Nanjing 210023, Jiangsu, P.R.China}
\email{zkyangrui2015@163.com}
\email{ecchen@njnu.edu.cn}
\email{zhouxiaoyaodeyouxian@126.com}
\date{}

\begin{document}
	
\maketitle

\renewcommand{\thefootnote}{}
\footnote{2020 \emph{Mathematics Subject Classification}: 37A05,  37A35, 37B40, 94A34.}
\footnotetext{\emph{Key words and phrases}: Variational principle; Metric mean dimension; Mean dimension;  Rate distortion dimension; R$\bar{e}$nyi information dimension.}
\footnote{*corresponding author}

\begin{abstract}
	
   Firstly, we  answer the problem 1  asked by Gutman and  $\rm \acute{\ S}$piewak in \cite{gs20}, then we establish a double  variational  principle for mean dimension in terms of R$\bar{e}$nyi information dimension and show   the order of $\sup$ and $\limsup$ (or $\liminf$) of the  variational principle for  the  metric mean dimension   in terms of   R$\bar{e}$nyi information dimension obtained by  Gutman and  $\rm \acute{\ S}$piewak  can be  changed under the marker property.  Finally, we attempt to  introduce the notion  of  maximal metric mean dimension measure, which is an analogue of the concept called classical maximal entropy measure related to the topological entropy.
\end{abstract}

\section{Introduction}

    By a pair $(\XX,T)$ we mean a topological dynamical system (TDS for short), where $\XX$ is a  compact metrizable  topological space  and $T: \XX \rightarrow \XX$ is a continuous self-map. The set of metrics on $\XX$  compatible with the topology is denoted by $\mathscr D(\XX)$. 
    By $M(\XX),M(\XX,T),E(\XX,T)$  we denote  the sets of all Borel probability measures on $\XX$, all $T$-invariant Borel probability measures on $\XX$, all ergodic measures on $\XX$, respectively.
    
    Mean topological dimension introduced  by Gromov \cite{gromov} is a  new  topological  invariant  in  topological dynamical systems, which was  systematically studied by  Lindenstrauss and Weiss \cite{lw00}. They also  introduced a notion called metric mean dimension  to capture the  topological complexity of infinite topological entropy systems and revealed  a  well-known result that  metric mean dimension  is an upper bound of mean topological dimension.  
    
    One says that a compact metric space  $(\XX, d)$ admits  \emph{tame growth of covering numbers} if for each $\theta>0$, $$\lim_{\epsilon \to 0}\epsilon^{\theta}\log r_1(T,d,\epsilon,\XX)=0,$$
    where $ r_1(T,d,\epsilon,\XX)$ denotes the minimal cardinality of  $\XX$ covered by   the balls $B(x,\epsilon):=\{y\in \XX:d(x,y)<\epsilon\}$.
    
    In 2018,  to inject ergodic
    theoretic ideas into mean dimension theory, Lindenstrauss  and Tsukamoto  \cite{lt18} established a  variational principle for metric mean dimension  in terms of rate distortion dimensions. Namely,
    
    \begin{thm1}\label{thm A}
    Let $(\XX,T)$ be a TDS with a metric $d$. Then
    \begin{align*}
    	\overline{{mdim}}_{M}(T,\XX,d)=\limsup_{\epsilon \to 0}\frac{\sup_{\mu \in M(\XX,T)}R_{\mu,L^{\infty}}(\epsilon)}{\logf}.
    \end{align*}
   Additionally, if $(\XX,d)$ has tame growth of covering numbers, then   for $p\in[1,\infty)$,

  \begin{align*}
	\overline{{mdim}}_{M}(T,\XX,d)=\limsup_{\epsilon \to 0}\frac{\sup_{\mu \in M(\XX,T)}R_{\mu,p}(\epsilon)}{\logf},
   \end{align*}
  where $\over$  denotes upper metric mean dimension of $\XX$, $R_{\mu,p}(\epsilon),R_{\mu,L^{\infty}}(\epsilon)$ are called the $L^p$  and  $L^{\infty}$ rate distortion function, respectively. Moreover, the above  two results  are  valid for  lower metric mean dimension $\underline{{mdim}}_{M}(T,\XX,d)$ by changing $\limsup$ into $\liminf$.
   \end{thm1}
    
   The  amenable version of Theorem A was proved by Chen et al. \cite{cdz22} by using abundant non-trivial  quasi-tiling methods. Besides,
   Gutman and  $\rm \acute{S}$piewak \cite{gs20} showed that the second  variational principle in Theorem A  can  only range  over all ergodic measures and
    posed   a problem  \cite[Problem 1]{gs20}  if the first  variational principle in Theorem A can only range over all ergodic measures. In this paper, we  give  a positive  answer to this problem.
  \begin{thm}\label{thm 1.1}
   Let $(\XX,T)$ be a TDS with a metric $d$.Then 
   	\begin{align*}
   		\over&=\limsup_{\epsilon \to 0}\frac{\sup_{\mu \in E(\XX,T)}R_{\mu,L^{\infty}}(\epsilon)}{\logf},\\
   		\under&=\liminf_{\epsilon \to 0}\frac{\sup_{\mu \in E(\XX,T)}R_{\mu,L^{\infty}}(\epsilon)}{\logf}.
   	\end{align*}
   \end{thm}
    
   Given $\mu \in M(\XX,T)$,  recall that lower and upper R$\bar{e}$nyi information dimensions of $\mu$ are respectively given by 
    \begin{align*}
     	\underline{MRID}(\XX,T,\mu,d)&=\liminf_{\epsilon \to 0}\frac{1}{\logf}\inf_{\diam P \leq \epsilon}h_\mu(T,P),\\
    	\overline{MRID}(\XX,T,\mu,d)&=\limsup_{\epsilon \to 0}\frac{1}{\logf}\inf_{\diam P \leq \epsilon}h_\mu(T,P),
    \end{align*}
   where the infimums range over all  finite partitions of $\XX$ with diameter at most $\epsilon$, and $h_\mu(T,P)$ denotes the measure-theoretic entropy  of $\mu$ with respect  to  $T$ and $P$,  its precise definition  is given in  section 2.

    One says that a  topological dynamical system $(\XX, T)$ admits \emph{marker property}  if for any $N>0$ there exists an open set $U\subset \XX$  with property that  
   $$U\cap T^n U=\emptyset, 1\leq n\leq N,  \mbox{and}~\XX=\cup_{n\in \mathbb{Z}} T^n U.$$
   
    After establishing variational principle  for metric mean dimension,  Lindenstrauss and Tsukamoto \cite{lt19}  further established  a double variational principle for mean dimension  in terms of $L^1$-rate distortion dimension under marker property, which can be  considered an important link between mean dimension  theory and ergodic theory. Later, Tsukamoto \cite{t20} extended this result to mean dimension with potential.  In this  paper,   replacing $L^1$-rate distortion dimension by  R$\bar{e}$nyi information dimension, we also  establish a double  variational  principle for mean dimension in terms of R$\bar{e}$nyi information dimension and  show   the order of $\sup$ and $\limsup$ (or $\liminf$) of the  variational principle for  the  metric mean dimension   in terms of   R$\bar{e}$nyi information dimension obtained by  Gutman and  $\rm \acute{\ S}$piewak  can be  changed under marker property.
     \begin{thm}\label{thm 1.3}
     Let $(\XX,T)$ be a TDS admitting marker property. Then  
      \begin{align*}
      	\mdim(\XX,T)&=\min_{d\in \mathscr{D}(\XX)} \sup_{\mu \in M(\XX,T)}\underline{MRID}(\XX,T,\mu,d),\\
      	&=\min_{d\in \mathscr{D}(\XX)} \sup_{\mu \in M(\XX,T)}\overline{MRID}(\XX,T,\mu,d),
      \end{align*}
      and  for any $d\in \mathscr{D}^{'}(\XX)$,  
      \begin{align*}
      	\under&=\over\\
      	&=\sup_{\mu \in M(\XX,T)}\liminf_{\epsilon \to 0}\frac{1}{\logf} \inf_{\diam P \leq \epsilon} \limits h_\mu(T,P)\\
      	&=\sup_{\mu \in M(\XX,T)}\limsup_{\epsilon \to 0}\frac{1}{\logf} \inf_{\diam P \leq \epsilon} \limits h_\mu(T,P)\\
      	&=\liminf_{\epsilon \to 0} \sup_{\mu \in M(\XX,T)}\frac{1}{\logf} \inf_{\diam P \leq \epsilon} \limits h_\mu(T,P)\\
      	&=\limsup_{\epsilon \to 0} \sup_{\mu \in M(\XX,T)}\frac{1}{\logf} \inf_{\diam P \leq \epsilon} \limits h_\mu(T,P),
      	\end{align*}
  $\mdim(\XX,T)$  denotes the mean dimension, see \cite[Subsection 1.2]{t20} for its explicit definition, and  $\mathscr{D}^{'}(\XX)=\{d\in \mathscr{D}(\XX): mdim(\XX,T)=\over\}$.   	
   \end{thm}
  We would like to remark that if a TDS admits marker property,  Tsukamoto \cite[Theorem 1.8]{t20} showed  there exists $d\in \mathscr{D}(\XX)$ such that  $mdim(\XX,T)=\over$, this implies $\mathscr{D}^{'}(\XX)$ is  not empty.

\section{Preliminary}
  
In this section,  we  recall the definition of metric mean  dimension and collect  several types of measure-theoretic entropies  for forthcoming proofs.

\subsection{Metric mean dimension}
Let $n\in \mathbb{N}$, 	for $x,y \in \XX$, we define the $n$-th Bowen metric $d_n$  on $\XX$ as 
$$d_n(x,y):=\max_{0\leq j\leq n-1}d(T^{j}(x),T^j(y)).$$

For a non-empty subset  $Z\subset \XX$. A set $E\subset \XX$ is  \emph{an $(n,\epsilon)$-\emph {spanning set} of $Z$} if  for any $x \in Z$, there  exists  $y\in E$ such that $d_n(x,y)<\epsilon.$ The smallest cardinality  of $(n,\epsilon)$-spanning set of $Z$ is denoted by $r_n(T,d,\epsilon,Z)$.
A set $F\subset Z$ is  \emph{an $(n,\epsilon)$-separated set of $Z$} if $d_n(x,y)\geq\epsilon$ for any $x,y \in F$ with $x\not= y$. The  largest  cardinality of $(n,\epsilon)$-separated set of $Z$  is denoted by $s_n(T,d,\epsilon,Z)$.

Put
$$r(T, \XX,d,\epsilon)=\limsup_{n\to \infty} \frac{1}{n} \log r_n(T,d,\epsilon, \XX)$$
and
$$s(T, \XX,d,\epsilon)=\limsup_{n\to \infty} \frac{1}{n} \log s_n(T,d,\epsilon,\XX).$$

By a standard method  used in \cite{w82}, we  have
$r(T, \XX,d,\epsilon)\leq s(T, \XX,d,\epsilon) \leq r(T, \XX,d,\frac{\epsilon}{2}).$

Let $(\XX,T)$ be a TDS, we  define \emph{the  upper metric mean dimension}  of $T$ on $\XX$  as
$$\overline{mdim}_{M}(T,\XX,d):=\limsup_{\epsilon\to 0}\frac{r(T,\XX,d,\epsilon)}{\log \frac{1}{\epsilon}}= \limsup_{\epsilon\to 0}\frac{s(T,\XX,d,\epsilon)}{\log \frac{1}{\epsilon}}.
$$

Similarly, one can define \emph{the lower metric mean dimension} $\under$ by  replacing  \emph{$\limsup$} with  \emph{$\liminf$}. If $\under=\over$, we call the common value   denoted by $mdim_M(T,\XX,d)$ \emph{metric mean dimension}.

\subsection{Measure-theoretical entropy}
Let $P$ be a   partition of $\XX$ and $\mu \in M(\XX,T)$, then the partition entropy of $P$ is given by 

$$H_\mu(P)=\sum_{A\in P}-\mu(A)\log \mu(A),$$  
here we use the convention that $\log=\log_e$ and $0\cdot \infty=0$.

Let $P, Q$ be two finite partitions of $\XX$, the join of $P$ and $Q$ is defined by $P \vee Q:=\{A\cap B: A\in P, B \in Q\}$.
Set $P^n:=\vee_{j=0}^{n-1}T^{-j}P$, and we define the the measure-theoretic entropy of $\mu$ with respect  to  $T$ and $P$  as
$$h_\mu(T,P)=\lim_{n\to \infty}\frac{1}{n}H_\mu(P^n).$$
The measure-theoretic entropy  of $\mu$ is defined by
$$h_\mu(T)=\sup_{P}h_\mu(T,P),$$
where $P$ ranges over all finite partitions of $\XX$.
\subsection{Rate distortion theory}

  Let $(\Omega,\mathbb{P})$ be a  probability space and $\XX$ and $\mathcal{Y}$ be  measurable spaces, and  let $\xi:\Omega \rightarrow \XX$ and $\eta: \Omega \rightarrow \mathcal{Y}$ be  measurable maps. We define the $\emph{mutual information}$ $I(\xi;\eta)$  as the supremum of 
 \begin{align*}
 	\sum_{1\leq m\leq  M,\atop 1\leq n \leq  N}\mathbb{P}(\xi \in P_m,\eta \in Q_n)\log \frac{\mathbb{P}(\xi \in P_m,\eta \in Q_n)}{\mathbb{P}(\xi \in P_m)\mathbb{P}(\eta \in Q_n)},
 \end{align*}
where $\{P_1,...,P_M\}$ and $\{Q_1,...,Q_N\}$  are partitions of $\XX$  and $\mathcal{Y}$ respectively, here we use  the convention that $0\log \frac{0}{a}=0$ for all $a\geq0$.

A measurable map  $\xi: \Omega \rightarrow \XX$ with finite image  naturally associates a  finite partition on $\Omega$, the $\emph{preimage partition}$ of $\Omega$. In this case we denote the entropy of  $\xi$ by $H(\xi)$.

If  $\XX$ and $\mathcal{Y}$ are finite sets, then $I(\xi;\eta)$ is given by 
\begin{align*}
  &\sum_{x\in\XX,y\in \mathcal{Y}}\mathbb{P}(\xi =x,\eta =y)\log \frac{\mathbb{P}(\xi =x,\eta =y)}{\mathbb{P}(\xi =x)\mathbb{P}(\eta =y)},\\
&=H(\xi)-H(\xi|\eta)=H(\xi)+H(\eta)-H(\xi\vee \eta),
\end{align*}
where $H(\xi|\eta)$ is  the conditional entropy of $\xi$ given  $\eta$, the value shows  the amount of information which  the random variables $\xi$ and $\eta$ share.

Let $(\XX,T)$ be a TDS  with a metric $d$ and $\mu \in M(\XX,T)$. Given $1\leq p<\infty$ and $\epsilon >0$, 
we define the  $\emph{$L^{p}$-rate distortion function}$  $R_{\mu,p}(\epsilon)$ as the infimum of
 $$\frac{I(\xi;\eta)}{n},$$
 where $n$ runs over  all natural numbers, and $\xi$ and $\eta=(\eta_0,...,\eta_{n-1})$  are random variables  defined on   some probability space $(\Omega, \mathbb{P})$ such that
 \begin{enumerate}
 	\item $\xi$ takes values in $\XX$, and its law is given by $\mu$.
 	\item Each $\eta_k$ takes  values in $\XX$, and 
 	\begin{align*}
 		\mathbb{E}\left(\frac{1}{n}\sum_{k=0}^{n-1}d(T^k\xi,\eta_k)^p\right)<\epsilon^p,
 	\end{align*}
 \end{enumerate} 
 where $\mathbb{E}(\cdot)$  is  the  expectation with respect to  the probability  measure $\mathbb{P}$.

Let $s>0$,
we define the  $\emph{$L^\infty$-rate distortion function}$  $R_{\mu,L^\infty}(\epsilon,s)$ as the infimum of
$$\frac{I(\xi;\eta)}{n},$$
where $n$ runs over  all natural numbers, and $\xi$ and $\eta=(\eta_0,...,\eta_{n-1})$  are random variables  defined on   some probability space $(\Omega, \mathbb{P})$ such that
\begin{enumerate}
	\item $\xi$ takes values in $\XX$, and its law is given by $\mu$.	
	\item Each $\eta_k$ takes  values in $\XX$, and 
	\begin{align*}
	\mathbb{E}\left( \text{the number of } 0\leq k\leq n-1~\text{with}~d(T^k\xi,\eta_k)\geq \epsilon\right)<sn,
	\end{align*}	
	 
\end{enumerate} 
and we set $R_{\mu,L^{\infty}}(\epsilon)=\lim_{s\to 0}\limits R_{\mu,L^{\infty}}(\epsilon,s)$.

In above two  definitions,  Lindenstrauss  and Tsukamoto \cite[Section IV, reamrk 14]{lt18}  showed that  the infimums can  only consider  random variable $\eta$ taking  finitely  many values,  and they  also showed \cite[III,B]{lt18}  that  for $1\leq p< \infty$, $R_{\mu,p}(\epsilon)\leq R_{\mu,L^\infty}(\epsilon^{'})$ holds for $0<\epsilon^{'}\leq\epsilon$.

We define the lower and upper $L^P$ rate distortion dimensions as 
\begin{align*}
\underline{rdim}_{L^p}(\XX,T,d,\mu)&=\liminf_{\epsilon \to 0}\frac{R_{\mu,p}(\epsilon)}{\logf}\\
\overline{rdim}_{L^p}(\XX,T,d,\mu)&=\limsup_{\epsilon \to 0}\frac{R_{\mu,p}(\epsilon)}{\logf}.
\end{align*}
 Replacing   $R_{\mu,p}(\epsilon)$ with $R_{\mu,L^{\infty}}(\epsilon)$, one can similarly define  lower and upper $L^{\infty}$ rate distortion dimensions.
 
 \subsection{Brin-Katok local entropy} 
 Measure-theoretic entropy is  given from the  viewpoint  of  the local perspective.
 Let  $\mu \in \mathcal{M}(X)$.
 Define 
 
 \begin{align*}
 	\overline{h}_{\mu}^{BK}(T, \epsilon):&=\int \overline{h}_{\mu}^{BK}(T, x,\epsilon)d\mu	\\
 	\underline{h}_{\mu}^{BK}(T, \epsilon):&=\int \underline{h}_{\mu}^{BK}(T, x,\epsilon)d\mu,
 \end{align*}
 where 
 \begin{align*}
 	\overline{h}_{\mu}^{BK}(T, x,\epsilon)=\limsup_{n\to \infty}-\frac{\log \mu (B_n(x,\epsilon))}{n}\\
 	\underline{h}_{\mu}^{BK}(T, x,\epsilon)=\liminf_{n\to \infty}-\frac{\log \mu (B_n(x,\epsilon))}{n}.
 \end{align*}

 If $\mu \in M(\XX,T)$,  Brin and Katok \cite{k80} showed 
 $$\lim_{\epsilon\to 0}\underline{h}_{\mu}^{BK}(T, \epsilon)=\lim_{\epsilon\to 0}\overline{h}_{\mu}^{BK}(T, \epsilon)=h_\mu(T).$$

\subsection{Katok's entropy}
Measure-theoretic entropy  defined by spanning set.

Let $\mu \in  \mathcal{M}(\XX)$, $\epsilon>0$ and $n \in\N$.  Given $\delta \in (0,1)$ and
put
$$R_\mu^\delta(T,n, \epsilon):=\min\{\#E: \mu (\cup_{x\in E}B_n(x,\epsilon))> 1-\delta \}.$$
Define 
\begin{align*}
	\overline{h}_{\mu}^K(T,\epsilon, \delta)&=\limsup_{n\to \infty} \frac{1}{n} \log R_\mu^\delta(T,n, \epsilon)\\
	\underline{h}_{\mu}^K(T,\epsilon, \delta)&=\liminf_{n\to \infty} \frac{1}{n} \log R_\mu^\delta(T,n, \epsilon). 
\end{align*}
If $\mu \in E(\XX,T)$, Katok \cite{k80} showed  
$$\lim_{\epsilon \to 0}\overline{h}_{\mu}^K(T,\epsilon, \delta)=\lim_{\epsilon \to 0}\underline{h}_{\mu}^K(T,\epsilon, \delta)=h_{\mu}(T).$$

We can define two quantities related to Katok's entropy by an alternative way. 

Let $\mu \in M(\XX)$. Note that the quantities $\overline{h}_{\mu}^K(T,\epsilon, \delta),\underline{h}_{\mu}^K(T,\epsilon, \delta)$ are  non-decreasing  when $\delta$ decreases. Therefore, we define 

\begin{align*}
	\overline{h}_{\mu}^K(T,\epsilon):=\lim_{\delta \to 0}\overline{h}_{\mu}^K(T, \epsilon, \delta),~~
	\underline{h}_{\mu}^K(T,\epsilon):=\lim_{\delta \to 0}\underline{h}_{\mu}^K(T,\epsilon, \delta).
\end{align*}

\subsection{Pfister and Sullivan's  entropy}
Measure-theoretic entropy  defined by separated set.
Let $\mu \in E(\XX,T)$ and $\epsilon>0$.
Define 
$$PS_\mu(T,d,\epsilon):=\inf_{F\ni \mu}\limsup_{n \to \infty}\frac{1}{n}\log s_n(T,d,\epsilon,\XX_{n,F}),$$
where  the infimum ranges over all neighborhoods of $\mu$ in $M(\XX)$ and  $\XX_{n,F}:=\{x\in \XX: \EE_n(x)=\frac{1}{n}\sum_{j=1}^{n-1} \delta_{T^{j}(x)} \in F\}$.
In fact, the  infimum can  only  range over any base of  open neighborhoods of $\mu$.

 Pfister and Sullivan  \cite{ps07} proved that $h_{\mu}(T)=\lim_{\epsilon \to 0}\limits PS_\mu(T,d,\epsilon)$.

\subsection{Generic points}
Let $\mu \in M(\XX,T)$, by 
$$G_\mu:=\{x\in \XX: \frac{1}{n}\sum_{j=0}^{n-1}\delta_{T^{j}(x)}\to \mu, n \to \infty\}$$   we denote  the set of generic points of $\mu$. Note that if  $\mu \in E(\XX,T)$, then we know  $\mu(G_\mu)=1$ by Birkhoff's  ergodic theorem.

\section{Proofs of  main results}
 
In this section, we prove Theorem \ref{thm 1.1}  and Theorem \ref{thm 1.3}.

Firstly, we  give the proof of Theorem 1.1.

The following lemma slightly modifies the statement in \cite[Lemma 4.1,(2)]{w21}, which  plays a key role in  the proof of Theorem 1.1.

\begin{lem}\label{lem  3.1}
Let $(\XX,T)$ be a TDS with a metric $d$ and $\{E_n\}_{n\geq 1}$ be a sequence non-empty  subsets of $\XX$. Let $\epsilon >0$, and let $F_n$ be an  $(n,6\epsilon)$-separated set  of $E_n$ with maximal cardinality $s_n(T,d,6\epsilon,E_n)$.  Set

$$\mu_n=\frac{1}{n\#F_n}\sum_{x\in F_n}\sum_{j=0}^{n-1}\delta_{T^{j}(x)}.$$
Choose  a subsequence $n_j$ such that $\mu_{n_j} $ convergences to $\mu\in M(\XX,T)$ in the weak* topology,
then
$$\limsup_{j\to \infty}\frac{1}{n_j}\log s_{n_j}(T,d,6\epsilon,E_{n_j})\leq R_{\mu,L^{\infty}}(\epsilon).$$
\end{lem}
\begin{proof}
This lemma can be proved by repeating the proof of \cite[Proposition 35]{lt18}.
\end{proof}
For sake of readers, we slightly modify the statement in \cite[Proposition 4.3]{w21} and repeat the proof.

\begin{lem}\label{lem 3.2}
Let $(\XX,T)$ be a TDS with a metric $d$. Then  for any $\epsilon>0$ and $\mu\in E(\XX,T)$, we have 
$$ \overline{h}_{\mu}^K(T,\epsilon)\leq PS_\mu(T,d,\epsilon)\leq R_{\mu,L^{\infty}}(\frac{1}{6}\epsilon).$$
\end{lem}
\begin{proof}
Given  $\epsilon>0$ and let $\mu  \in E(\XX,T)$.  Fix a base of  open neighborhoods $\mathcal F_\mu$ of $\mu$, if  $F \in \mathcal F_\mu$, then $G_\mu\subset \cup_{N\geq1}\cap_{n\geq N}\XX_{n,F}$. Let $\delta \in(0,1)$ and note that $\mu(G_\mu)=1$, we can  find $N_0$ such that for any $n\geq N_0$, $\mu(\XX_{n,F})> 1-\delta$. Hence, $R_\mu ^\delta (T,n,\epsilon)\leq s_n(T,d,\epsilon,\XX_{n,F})$ for any $n\geq N_0$. This implies that $	\overline{h}_{\mu}^K(T,\epsilon, \delta)\leq \limsup_{n \to \infty}\frac{1}{n}\log s_n(T,d,\epsilon,\XX_{n,F})$ holds for any $\delta \in (0,1)$. Letting $\delta \to 0$ and by the arbitrariness of $F$, we get $\overline{h}_{\mu}^K(T,\epsilon)\leq PS_\mu(T,d,\epsilon)$.

Next, given  $\epsilon>0$ and let $\mu  \in E(\XX,T)$ again, we show  $PS_\mu(T,d,6\epsilon)\leq R_{\mu,L^{\infty}}(\epsilon).$  Without loss of generality, we may assume that  $R_{\mu,L^{\infty}}(\epsilon)<\infty$. If
$$PS_\mu(T,d,6\epsilon)=\inf_{F\ni \mu}\limsup_{n \to \infty}\frac{1}{n}\log s_n(T,d,6\epsilon,\XX_{n,F})> R_{\mu,L^{\infty}}(\epsilon),$$
then we can choose $\gamma_0>0$ and a  decreasing  sequence of  closed convex  neighborhood $\{C_n\}$ of $\mu$ such that 
\begin{align}\label{equ 3.1}
\limsup_{n \to \infty}\frac{1}{n}\log s_n(T,d,6\epsilon,\XX_{n,C_n})> R_{\mu,L^{\infty}}(\epsilon)+\gamma_0
\end{align}
and $\cap_{n\geq 1}{C_n}=\{\mu\}$.

Let $F_n\subset \XX_{n,C_n}$  be an $(n,6\epsilon)$-separated set of $\XX_{n,C_n}$ with the maximal cardinality  $ s_n(T,d,6\epsilon,\XX_{n,C_n})$.  Set 
$$\mu_n=\frac{1}{n\#F_n}\sum_{x\in F_n}\sum_{j=0}^{n-1}\delta_{T^{j}(x)}.$$
Then $\mu_n \in C_n$ and $\lim\limits_{n\to \infty}\mu_n=\mu$. 

By Lemma \ref{lem  3.1}, we know that  $\limsup_{n \to \infty}\frac{1}{n}\log s_n(T,d,6\epsilon,\XX_{n,C_n})\leq R_{\mu,L^{\infty}}(\epsilon)$, which contradicts  with  the  inequality  (\ref{equ 3.1}). 

\end{proof}

\begin{thm}\label{thm 3.3}
	Let $(\XX,T)$ be a TDS with a metric $d$. Then
	\begin{align*}
		\over&=\limsup_{\epsilon \to 0}\frac{\sup_{\mu \in E(\XX,T)}\underline h_\mu^K(T,\epsilon)}{\logf}\\
		&=\limsup_{\epsilon \to 0}\frac{\sup_{\mu \in E(\XX,T)}\overline h_\mu^K(T,\epsilon)}{\logf}.
	\end{align*}
The two   variational principles are  valid for   $\under$ by changing $\limsup$ into $\liminf$ and the supremum  can   range  over all  invariant measures.
\end{thm}

\begin{proof}
Fix $\delta_0 \in (0,1)$, then we have 
	\begin{align*}
	\over&=\limsup_{\epsilon \to 0}\frac{\sup_{\mu \in E(\XX,T)}\underline h_\mu^K(T,\epsilon,\delta_0)}{\logf}, \text{ by \cite[Proposition 7.3]{shi}}\\
	&\leq\limsup_{\epsilon \to 0}\frac{\sup_{\mu \in E(\XX,T)}\underline h_\mu^K(T,\epsilon)}{\logf}\\
	&\leq\limsup_{\epsilon \to 0}\frac{\sup_{\mu \in E(\XX,T)}\overline h_\mu^K(T,\epsilon)}{\logf}.	
\end{align*}
On the other hand,  fix $\epsilon >0$ and $\mu \in M(\XX)$. Let $\delta \in (0,1)$, then $R_\mu^\delta (T,n,\epsilon)\leq r_n(T,d,\epsilon,\XX)$ for  every $n\in \mathbb{N}$, which yields that  $\overline{h}_{\mu}^K(T,\epsilon, \delta)\leq r(T,\XX,d,\epsilon)$  holds for every $\delta \in (0,1).$ Letting $\delta \to 0$ gives $\overline{h}_{\mu}^K(T,\epsilon)\leq r(T,\XX,d,\epsilon)$. Hence,  we finally obtain that 
$$\limsup_{\epsilon \to 0}\frac{\sup_{\mu \in E(\XX,T)}\overline h_\mu^K(T,\epsilon)}{\logf}\leq \over.$$
This completes the proof.
\end{proof}

Now, we are ready to give the proof of Theorem \ref{thm 1.1}.
\begin{proof}[Proof of Theorem \ref{thm 1.1}]
	It suffices to show  show the first equality, and the second one can be obtained in a similar manner.
	By Theorem  A,   it is clear that  
	$$	\over\geq\limsup_{\epsilon \to 0}\frac{\sup_{\mu \in E(\XX,T)}R_{\mu,L^{\infty}}(\epsilon)}{\logf}.$$
	On the other hand, by Lemma \ref{lem 3.2} and Theorem  \ref{thm 3.3}, we get  the converse inequality. 
\end{proof}

Next, we proceed to   give the proof of Theorem \ref{thm 1.3}.

\begin{prop}\label{prop 3.4}
	Let $(\XX,T)$ be a TDS and $\mu \in M(\XX,T)$. Then  for every $\epsilon >0$ and $\mu \in M(\XX,T)$,  we have
	$$R_{\mu,L^{\infty}}(2\epsilon)\leq \inf_{\diam P \leq \epsilon}h_\mu(T,P),$$
	where the infimum ranges over all finite partitions of $\XX$.
\end{prop}

\begin{proof}
Fix $\epsilon >0$ and $\mu \in M(\XX,T)$,   we can  choose a finite partition $Q$  of   $\XX$ so that 
$\inf_{\diam P \leq \epsilon}\limits h_\mu(T,P)\leq\log\#Q<\infty$. Let   $P$ be a finite partition of $\XX$ with diameter at most $\epsilon$,  and let $\xi$ be a random variable taking values in $\XX$ and obeying $\mu$. For every  $n \in \mathbb{N}$ and  every $A\in P^n$, we choose $x_A\in A$ and define a map  $f: \XX \rightarrow \XX$ by setting  $f(x)=x_A$ if $x\in A$. Put
$\eta={(f(\xi),Tf(\xi),...,T^{n-1}f(\xi))}$, then
$$\mathbb{E}\left( \text{the number of } k\in [0,n-1]~\text{with}~d(T^k\xi,T^kf(\xi))\geq 2\epsilon\right)=0<sn,$$ for any $s>0$. Hence,
	$R_{\mu,L^{\infty}} (2\epsilon,s)\leq \frac{I(\xi;\eta)}{n}\leq\frac{H(\eta)}{n}=\frac{H_\mu(P^n)}{n}$, this implies that  $R_{\mu,L^{\infty}}(2\epsilon,s)\leq h_\mu(T,P)$. Letting $s \to 0$ gives the desired result.	
	
\end{proof}

\begin{proof}[Proof of Theorem \ref{thm 1.3}]
By  \cite[Theorem 3.1]{gs20}, we have 
\begin{align}\label{equ 3.2 }
	\over&=	\limsup_{\epsilon \to 0}\frac{\sup_{\mu \in M(\XX,T)}\limits \inf_{\diam P \leq \epsilon} \limits h_\mu(T,P)}{\logf}\\
	&\geq\sup_{\mu \in M(\XX,T)} \limsup_{\epsilon \to 0}\frac{1}{\logf}\inf_{\diam P \leq \epsilon} \limits h_\mu(T,P)\nonumber
\end{align}
By the fact obtained in \cite[III,B]{lt18}  that  for $1\leq p< \infty$, $R_{\mu,p}(\epsilon)\leq R_{\mu,L^\infty}(\epsilon^{'})$ holds for $0<\epsilon^{'}\leq\epsilon$, we have $R_{\mu,1}(3\epsilon)\leq R_{\mu,L^\infty}(2\epsilon)$. Together with  the Proposition  \ref{prop 3.4}, we obtain that 
for any $\mu \in M(\XX,T)$,
$$ \underline{rdim}_{L^1}(\XX,T,d,\mu)\leq \liminf_{\epsilon \to 0}\frac{1}{\logf}\inf_{\diam P \leq \epsilon} \limits h_\mu(T,P).$$
Therefore, for any $d\in \mathscr{D}(\XX)$,
\begin{align*}
mdim(\XX,T)&\leq \sup_{\mu \in M(\XX,T)} \underline{rdim}_{L^1}(\XX,T,d,\mu) ~\text{by  Corollary \cite[Corollary 1.7]{t20}}\\
 &\leq\sup_{\mu \in M(\XX,T)} \liminf_{\epsilon \to 0}\frac{1}{\logf}\inf_{\diam P \leq \epsilon} \limits h_\mu(T,P) \\
 &\leq \sup_{\mu \in M(\XX,T)}\limsup_{\epsilon \to 0}\frac{1}{\logf}\inf_{\diam P \leq \epsilon} \limits h_\mu(T,P)\\
 &\leq \over~~~~ \text{by ~(\ref{equ 3.2 })}.
\end{align*}
Using the fact obtained in \cite[Theorem 1.8]{t20}, if $(\XX,T)$ admits marker property, then there exists $d\in \mathscr{D}(\XX)$ such that $mdim(\XX,T)=\over$. 
This implies that 
\begin{align*}
\mdim(\XX,T)&=\min_{d\in \mathscr{D}(\XX)} \sup_{\mu \in M(\XX,T)}\liminf_{\epsilon \to 0}\frac{1}{\logf}\inf_{\diam P \leq \epsilon} \limits h_\mu(T,P),\\
&=\min_{d\in \mathscr{D}(\XX)} \sup_{\mu \in M(\XX,T)}\limsup_{\epsilon \to 0}\frac{1}{\logf}\inf_{\diam P \leq \epsilon} \limits h_\mu(T,P).
\end{align*}

By the Lindenstrauss and Weiss's  classical inequality \cite{lw00},  $\mdim(\XX,T) \leq \under
 \leq \over$  for any $d\in \mathscr{D}(\XX)$, then for any $d\in \mathscr{D}^{'}(\XX)$, we have  $\mdim(\XX,T)=\under=\over$.  We finally deduce that for any $d\in \mathscr{D}^{'}(\XX)$,
\begin{align*}
	\under&=\over\\
	&=\sup_{\mu \in M(\XX,T)}\liminf_{\epsilon \to 0}\frac{1}{\logf} \inf_{\diam P \leq \epsilon} \limits h_\mu(T,P)\\
	&=\sup_{\mu \in M(\XX,T)}\limsup_{\epsilon \to 0}\frac{1}{\logf} \inf_{\diam P \leq \epsilon} \limits h_\mu(T,P)\\
	&=\liminf_{\epsilon \to 0} \sup_{\mu \in M(\XX,T)}\frac{1}{\logf} \inf_{\diam P \leq \epsilon} \limits h_\mu(T,P)\\
	&=\limsup_{\epsilon \to 0} \sup_{\mu \in M(\XX,T)}\frac{1}{\logf} \inf_{\diam P \leq \epsilon} \limits h_\mu(T,P),
\end{align*}
where the last two equalities hold by \cite[Theorem 3.1]{gs20}.
\end{proof}

For  a TDS admits marker property,  Theorem 1.2 shows that for some ``nice" metrics, the variational principles  are still  valid  if we  change the order of $\sup$ and $\limsup$ (or $\liminf$).  It is  not clear that whether we can drop the marker property imposed on the topological dynamical system or not,  and removing the  condition  requires us to answer  a central problem  that if  for every topological dynamical system,   there exists a metric $d$ such that   $mdim(\XX,T)=\over$.
This open problem  was  also mentioned in \cite{glt16,lt19,t20}.

Finally, we attempt to introduce the notion of maximal metric mean dimension measure analogous to the  classical notion of maximal entropy measure related to topological entropy.  We begin this new concept with the  following example.
\begin{ex}
Let $\sigma:[0,1]^{\mathbb{Z}}\rightarrow [0,1]^{\mathbb{Z}}$ be the (left) shift on alphabet $[0,1]$, where $[0,1]$ is the unit interval with the standard metric. Equipped 
$[0,1]^{\mathbb{Z}}$ with a metric given by 
$$d(x,y)=\sum_{n\in \mathbb{Z}}2^{-|n|}|x_n-y_n|.$$
Let $\mu=\mathcal{L}^{\otimes \mathbb{Z}}$, where $\mathcal{L}$ is the Lebesgue measure on $[0,1]$. 

Let $\XX=[0,1]^{\mathbb{Z}}$ and $T=\sigma$. It is well-known that ${{mdim}}_M(T,\XX,d)=1$, see \cite[Section II, E. Example]{lt18} for more details.  Shi \cite[Exmaple 6.1]{shi} showed that $\lim_{\epsilon \to 0}\limits  \frac{\overline h_{\mu}^{BK}(T,\epsilon)}{\log \frac{1}{\epsilon}}=\lim_{\epsilon \to 0}\limits  \frac{\underline h_{\mu}^{BK}(T,\epsilon)}{\log \frac{1}{\epsilon}}=1$. By \cite[Proposition 3.1]{ycz22}, we know
$\lim_{\epsilon \to 0}\limits\frac{\inf_{\diam P \leq \epsilon}\limits h_\mu(T,P)}{\log \frac{1}{\epsilon}}=\lim_{\epsilon \to 0}\limits  \frac{\overline h_{\mu}^{BK}(T,\epsilon)}{\log \frac{1}{\epsilon}}=\lim_{\epsilon \to 0}\limits  \frac{\overline h_{\mu}^{BK}(T,\epsilon)}{\log \frac{1}{\epsilon}}=1$.
This shows that
\begin{align*}
	{{mdim}}_M(T,\XX,d)&=\lim_{\epsilon \to 0}\limits\frac{\inf_{\diam P \leq \epsilon}\limits h_\mu(T,P)}{\log \frac{1}{\epsilon}}\\
	&=\lim_{\epsilon \to 0}\limits  \frac{\overline h_{\mu}^{BK}(T,\epsilon)}{\log \frac{1}{\epsilon}}=\lim_{\epsilon \to 0}\limits \frac{\underline h_{\mu}^{BK}(T,\epsilon)}{\log \frac{1}{\epsilon}}=1.
\end{align*}

\end{ex}
Unlike the measure-theoretical entropy used  to establish variational principle  for topological entropy,   there are  abundant  choices that can be considered as a object to  establish a  variational principle for metric mean dimension \cite{shi}, for example Brin-Katok local entropy, Katok entropy,...  Therefore, to inject ergodic theoretic ideas into mean dimension theory,  a  reasonable quantity related to measure-theoretical  entropy is crucial. This leads to the following

  \begin{df}
A non-negative  real-valued function $F(\mu,\epsilon)$  defined on  $M(\XX,T)\times \mathbb{R}_+$ (or $E(\XX,T)\times \mathbb{R}_+$) is said to be a \emph{candidate}  if for any fixed $\mu$, $F(\mu,\epsilon)$ is non-decreasing  as $\epsilon$  decreases and $h_\mu(T)=\lim_{\epsilon \to 0}\limits F(\mu,\epsilon)$,  and we define  \emph{upper measure-theoretical metric mean dimension of $\mu$} as  $$\overline{mdim}_M(T,\XX,\mu)=\limsup_{\epsilon \to 0}\limits  \frac{F(\mu,\epsilon)}{\logf}.$$
  \end{df}

	  Similarly, one can  define lower measure-theoretical metric mean dimension $ \underline{mdim}_M(T,\XX,\mu)$ by replacing $\limsup_{\epsilon \to 0}\limits$ by  $\liminf_{\epsilon \to 0}\limits$. If $\underline{mdim}_M(T,\XX,\mu)\\=\overline{mdim}M(T,\XX,\mu)$, we call the common value \emph{measure-theoretical metric mean dimension of $\mu$}. 
   Given  $\mu \in E(\XX,T)$, such candidates  can be  $\overline{h}^{K}_\mu(T,\epsilon,\delta)$, $\underline{h}^{K}_\mu(T,\epsilon,\delta)$,  $\overline{h}_{\mu}^{BK}(T, \epsilon)$, $\underline{h}_{\mu}^{BK}(T, \epsilon)$ $R_{\mu,L^{\infty}}(\epsilon)$, readers can turn to    \cite[Subsection 2.2]{ycz22} for more candidates.

The quantity   (or we refer to ``speed")   $\overline{mdim}(T,\XX,\mu)$ can be interpreted  as  how fast the  candidate $F(\mu,\epsilon)$ approximate the (infinite) measure-theoretical entropy $h_\mu(T)$ as $\epsilon \to 0$. Namely, when $\epsilon>0$ is sufficiently small, we may  approximate $F(\mu,\epsilon)$  as 
$$F(\mu,\epsilon) \approx \overline{mdim}_M(T,\XX,\mu)\logf.$$

\begin{df}
	Let $(\XX,T)$ be a TDS and $d\in  \mathscr{D}(\XX)$.
	Given a candidate $F(\mu,\epsilon)$ satisfying $$\over=\sup_{\mu \in M(\XX,T)}\overline{mdim}_M(T,\XX,\mu),$$
	and we call $\mu$  a \emph{maximal upper (resp. lower) metric mean dimension measure} if $\over=\overline{mdim}_M(T,\XX,\mu)$ (resp. $\under=\overline{mdim}_M(T,\mu,\XX,\mu)$). The set of  all maximal upper (resp. lower) metric mean dimension measures is denoted by 
	$\overline{M}_{max}(T,\XX,d)$ (resp.  $\underline{M}_{max}(T,\XX,d)$).
\end{df}

Obviously,    $\overline{M}_{max}(T,\XX,d)$ and  $\underline{M}_{max}(T,\XX,d)$   depend on the metric  and the candidate that we choose.

\begin{prop}
Let $(\XX,T)$ be a TDS and $d\in  \mathscr{D}(\XX)$.
Given a candidate $F(\mu,\epsilon)$ satisfying
$$\over=\sup_{\mu \in M(\XX,T)}\limsup_{\epsilon \to 0}\limits  \frac{F(\mu,\epsilon)}{\logf}$$
and
$$\under=\sup_{\mu \in M(\XX,T)}\liminf_{\epsilon \to 0}\limits  \frac{F(\mu,\epsilon)}{\logf}.$$
Then the following statements hold
\begin{enumerate}
\item If $\under=\over$, then $\underline{M}_{max}(T,\XX,d)\subset \overline{M}_{max}(T,\XX,d)$.
\item  If  for any $\mu \in \overline{M}_{max}(T,\XX,d)$  satisfies $\limsup_{\epsilon \to 0}\limits  \frac{F(\mu,\epsilon)}{\logf}=\liminf_{\epsilon \to 0}\limits  \frac{F(\mu,\epsilon)}{\logf}$, then  $\underline{M}_{max}(T,\XX,d)\supset \overline{M}_{max}(T,\XX,d).$
\item  If for every  fixed $\epsilon>0$,  $F(\mu,\epsilon)$ is a concave function on $M(\XX,T)$ and  $\underline{M}_{max}(T,\XX,d)\not=\emptyset$, then the set $\underline{M}_{max}(T,\XX,d)$ is a convex subset of $M(\XX,T)$. Additionally, if $\under=\infty$, then $\underline{M}_{max}(T,\XX,d)\not=\emptyset$.
\end{enumerate}

\end{prop}

\begin{proof}
	(1) Without loss of generality, we assume that $\underline{M}_{max}(T,\XX,d)\not= \emptyset$.  Let $\mu \in \underline{M}_{max}(T,\XX,d)$, then 
	\begin{align*}
		\under&=\liminf_{\epsilon \to 0}\limits  \frac{F(\mu,\epsilon)}{\logf}\\
		&\leq \limsup_{\epsilon \to 0}\limits  \frac{F(\mu,\epsilon)}{\logf}\leq \over.
	\end{align*}
This implies that  $\underline{M}_{max}(T,\XX,d)\subset \overline{M}_{max}(T,\XX,d)$. 

(2)Let $\mu \in \overline{M}_{max}(T,\XX,d)$  with  $\limsup_{\epsilon \to 0}\limits  \frac{F(\mu,\epsilon)}{\logf}=\liminf_{\epsilon \to 0}\limits  \frac{F(\mu,\epsilon)}{\logf}$, then 
	\begin{align*}
	\over&=\limsup_{\epsilon \to 0}\limits  \frac{F(\mu,\epsilon)}{\logf}= \liminf_{\epsilon \to 0}\limits  \frac{F(\mu,\epsilon)}{\logf}\\
	&\leq \under\leq \over,
\end{align*}
 which yields that $\underline{M}_{max}(T,\XX,d)\supset \overline{M}_{max}(T,\XX,d).$
 
(3)  Let $\underline{mdim}_M(T,\XX,\mu)=\liminf_{\epsilon \to 0}\limits  \frac{F(\mu,\epsilon)}{\logf}$. Since $F(\mu,\epsilon)$ is a concave function on $M(\XX,T)$ for any fixed  $\epsilon>0$,   then $\underline{mdim}_M(T,\XX,\mu)$ is also concave. Let $\mu_1,\mu_2\in\underline{M}_{max}(T,\XX,d)$  and $p \in [0,1]$, this yields that
\begin{align*}
	\under&=p\underline{mdim}_M(T,\XX,\mu_1)+(1-p)\underline{mdim}_M(T,\XX,\mu_2)\\
	&\leq \underline{mdim}_M(T,\XX,p\mu_1+(1-p)\mu_2)\leq \under.
\end{align*}
It follows that $p\mu_1+(1-p)\mu_2 \in \underline{M}_{max}(T,\XX,d)$, which shows $\underline{M}_{max}(T,\XX,d)$ is convex.

If $\over=\infty$, then for each $n\in \mathbb{N}$, we  can choose $\mu_n\in M(\XX,T)$ such that $\underline{mdim}_M(T,\XX,\mu_n)>2^n$.
Set $\mu=\sum_{n=1}^{\infty}\limits \frac{1}{2^n}\mu_n=\sum_{n=1}^{N}\limits \frac{1}{2^n}\mu_n+\frac{1}{2^N}\nu_N$ for every  $N\in \mathbb{N}$, where $\nu_N \in M(\XX,T)$. Using the  concavity  of  $\underline{mdim}_M(T,\XX,\mu)$ with respect  to $\mu$, we have 
\begin{align*}
	 \underline{mdim}_M(T,\XX,\mu)\geq \sum_{n=1}^N \frac{1}{2^n}\underline{mdim}_M(T,\XX,\mu_n)>N.
\end{align*}
Letting $N \to \infty$ gives $\underline{mdim}_M(T,\XX,\mu)=\infty$, which   shows that $\underline{M}_{max}(T,\XX,d)\\
\not=\emptyset$.

\end{proof}
We finally end up this paper with a question as follows.

 \textbf{Question 1}  For every  topological dynamical system $(\XX,T)$, can we choose    proper metric $d$ and  proper   candidate $F(\mu,\epsilon)$ such that there exists $\mu \in M(\XX,T)$(or $E(\XX,T)$)  satisfying
\begin{align*}
\over&=\limsup_{\epsilon \to 0}\limits  \frac{F(\mu,\epsilon)}{\logf}\\
\under&=\liminf_{\epsilon \to 0}\limits  \frac{F(\mu,\epsilon)}{\logf}?
\end{align*}

In other words, such a metric $d$ and $\mu$  have the same speed  that  respectively approximate  infinite topological entropy and infinite measure-theoretical entropy.

\section*{Acknowledgement} 

\noindent The work was supported by the
National Natural Science Foundation of China (Nos.12071222 and 11971236), China Postdoctoral Science Foundation (No.2016M591873),
and China Postdoctoral Science Special Foundation (No.2017T100384). The work was also funded by the Priority Academic Program Development of Jiangsu Higher Education Institutions.  We would like to express our gratitude to Tianyuan Mathematical Center in Southwest China(11826102), Sichuan University and Southwest Jiaotong University for their support and hospitality.


\end{document}